\documentclass[letterpaper, 10 pt, conference]{ieeeconf} 

\IEEEoverridecommandlockouts                              % This command is only needed if 
\usepackage{cite}
\usepackage{amssymb,amsfonts}
\usepackage{algorithmic}
\usepackage{mathtools}
\usepackage{graphicx}
\usepackage{textcomp}
\usepackage[dvipsnames]{xcolor}
\usepackage{hyperref}
\usepackage[printonlyused]{acronym}
\usepackage{lmodern}
\usepackage[utf8]{inputenc}
\usepackage{graphicx}
\usepackage[export]{adjustbox}
\usepackage{cases}
\usepackage{empheq}
\usepackage{fancyhdr}
\usepackage{rotating}
\usepackage{bigdelim}
\usepackage{bm}
\usepackage{bbm}
\usepackage{tikz}
\usepackage{etoolbox}
\usepackage[nocomma]{optidef}
\usepackage{multirow}
\usepackage{bookmark}
\usepackage{comment}

\newcommand{\secref}[1]{Sec.~\ref{#1}}
\newcommand{\lemref}[1]{Lemma~\ref{#1}}

\newcommand{\theoref}[1]{Thm~\ref{#1}}

\renewcommand{\eqref}[1]{(\ref{#1})}
\newcommand{\figref}[1]{Fig.~\ref{#1}}

\newcommand{\assumptionsref}[1]{Assumption~\ref{#1}}

\newtheorem{theorem}{Theorem}

\newtheorem{remark}{Remark}

\newtheorem{lemma}{Lemma}

\newtheorem{assumptions}{Assumption}
\renewcommand{\theenumi}{(\roman{enumi})}%

\newcommand{\R}{\mathbb{R}}

\newcommand{\Z}{\mathbb{Z}}
\newcommand{\N}{\mathbb{N}}

\newcommand{\sign}{\mathrm{sgn}}

\DeclareMathOperator*{\minimize}{\mathrm{minimize}}

\newcommand{\ift}{\mathrm{~if~}}

\newcommand{\andt}{\mathrm{~and~}}

\newcommand{\fort}{\mathrm{~for~}}

\usepackage{upgreek}
\usepackage[ruled,vlined,linesnumbered]{algorithm2e}
\renewcommand{\S}{\mathcal{S}}

\newcommand{\T}{\mathcal{T}}
\newcommand{\A}{\mathcal{A}}
\newcommand{\U}{\mathcal{U}}

\newcommand{\Y}{\mathcal{Y}}
\newcommand{\C}{\mathcal{C}}
\def\Z{\mathcal{Z}}
\def\N{\mathcal{N}}

\usepackage{acronym}
\newacro{pm}[PM]{Persistent Monitoring}
\newacro{tsp}[TSP]{Traveling Salesman Problem}
\newacro{pmp}[PMP]{Pontryagin Minimum Principle}
\newacro{lls}[LLS]{Linear Least Squares}
\newacro{ocp}[OCP]{Optimal Control Problem}
\newacro{ipa}[IPA]{Infinitesimal Perturbation Analysis}

\usepackage{booktabs}

\newcommand{\target}{\vartheta}
\newcommand{\agent}{s}

\newcommand{\switchpos}{\psi}
\newcommand{\trackinter}{\phi}
\newcommand{\trackcomb}{\alpha}
\newcommand{\Kp}{K_p}
\newcommand{\Ki}{K_i}

\newcommand{\eventtime}{\tau}
\newcommand{\param}{\theta}
\newcommand{\allparams}{\mathbf{\theta}}

\newcommand{\stepdirection}{p}

\newcommand{\dtarget}{\dot{\target}}

\usepackage{graphicx}

\newcounter{inlineenum}
\renewcommand{\theinlineenum}{\alph{inlineenum}}
\newenvironment{inlineenum}
  {\unskip\ignorespaces\setcounter{inlineenum}{0}%
   \renewcommand{\item}{\refstepcounter{inlineenum}{\textit{\theinlineenum})~}}}
  {\ignorespacesafterend}

\definecolor{agentcolor}{HTML}{326da8}
\definecolor{targetcolor}{HTML}{c84646}

\newcommand{\mindist}{\Delta^\mathrm{min}}

\newcommand{\figwidth}{\columnwidth}

\begin{document}

\title{Optimal Persistent Monitoring of Mobile Targets in One Dimension}

\author{\IEEEauthorblockN{}
\IEEEauthorblockA{\textit{Division of Systems Engineering} \\
\textit{Boston University}
}
}

\author{Jonas Hall$^{1}$, Sean Andersson$^{1,2}$, Christos G. Cassandras$^{1,3}$
\thanks{This work was supported in part by NSF under grants ECCS-
1931600, DMS-1664644, CNS-1645681, CNS-2149511, by AFOSR under FA9550-19-1-0158, by ARPA-E under DE-AR0001282, by NIH under 1R01GM117039-01A1, and by
the MathWorks.}%
\thanks{$^{1}$Division of Systems Engineering, Boston University, USA}%
\thanks{$^{2}$Division of Mechanical Engineering, Boston University, USA}%
\thanks{$^{3}$Department of Electrical and Computer Engineering}%
\thanks{{\tt\small \{hallj, sanderss, cgc\}@bu.edu}}
}%

\clearpage\maketitle
\thispagestyle{empty}

\begin{abstract}
    This work shows the existence of optimal control laws for persistent monitoring of mobile targets in a one-dimensional mission space and derives explicit solutions. The underlying performance metric consists of minimizing the total uncertainty accumulated over a finite mission time. We first demonstrate that the corresponding optimal control problem can be reduced to a finite-dimensional optimization problem, and then establish existence of an optimal solution. Motivated by this result, we construct a parametric reformulation for which an event based gradient descent method is utilized with the goal of deriving (locally optimal) solutions. We additionally provide a more practical parameterization that has attractive properties such as simplicity, flexibility, and robustness. Both parameterizations are validated through simulation.
\end{abstract}
\section{Introduction}
In persistent monitoring problems, a team of autonomous agents is tasked to monitor a given environment. The problem is closely related to coverage control~\cite{cortes2004coverage}, the key difference being that the environment is assumed to change dynamically, so that all points of interest must be revisited persistently. The problem has received a lot of attention over the last decade due to its broad applicability across a range of applications, including environmental monitoring~\cite{yu2015persistent}, data harvesting~\cite{khazaeni2017event}, and particle tracking~\cite{shen2010tracking}. Persistent monitoring problems are notoriously difficult to solve because of highly non-convex and non-smooth costs and dynamics along with the computational complexity due to the problem's combinatorial nature. 

In this paper we consider the persistent monitoring problem applied to a finite number of targets in a one-dimensional connected mission space, similar to the formulations in~\cite{smith2011persistent, zhou2018optimal}. One-dimensional mission spaces are common in applications that involve rivers, roads in a transportation network, or power lines, to name a few. Unlike previous work, however, in addition to an internal state describing a measure of uncertainty, we allow \emph{the targets themselves to be mobile} in the mission space. The control objective considered is the minimization of the overall uncertainty over a finite mission time. The uncertainty dynamics considered in this work can be viewed as a flow model with constant input flow, whereas the output flow depends on the target-agent distance. This model is commonly used in the literature~\cite{cassandras2012optimal,smith2011persistent}, though other dynamics have been discussed as well. For example, in~\cite{yu2019scheduling}, the authors consider internal states with arbitrary linear time-invariant dynamics, with a control applied only when an agent is monitoring the target. In other applications, such as environmental monitoring tasks, the internal state is modeled as a Gaussian random field~\cite{lan2013planning}. Similarly, the literature contains various performance metrics, e.g., that of maximizing the number of observed events~\cite{yu2015persistent}, minimizing the inter-observation time~\cite{alamdari2014persistent}, or minimizing estimation errors~\cite{lan2013planning}.

The persistent monitoring problem is often split into a higher and lower level. The higher level consists of path planning and scheduling target visiting sequences. Lan and Schwager proposed the use of rapidly-exploring random cycles in order to find an optimal periodic trajectory~\cite{lan2013planning}. Other formulations abstract the mission space to a graph, where each node represents a region of interest and the edges capture the travel times~\cite{alamdari2014persistent,welikala2021event}. The graph-based formulation was further generalized in~\cite{hari2019generalized}, in which targets were grouped into clusters and only a single target within each cluster must be visited. Due to the combinatorial nature of the scheduling task, the use of mixed-integer optimization is natural~\cite{maini2018persistent}, and its similarity to the Traveling Salesman Problem has motivated solving the higher level scheduler with such methods~\cite{ostertag2022trajectory}. On the lower level, the problem consists of optimizing the agents' motion control, either along a given path, or in order to realize a given schedule. For example, in~\cite{smith2011persistent} and \cite{song2014optimal}, the authors optimize agent velocities along a given path, while in \cite{boldrer2022multi} a configuration-based solution is sought such that each point on a given path is revisited with a constant frequency.  

The contribution of this paper consists of advances in the one-dimensional persistent monitoring problem of mobile targets. In particular, we prove that the considered infinite-dimensional \ac{ocp} can be reduced to an optimization problem of finite dimension, and that the \ac{ocp} always admits a solution. Though similar results can be found for static targets, we provide novel (to the best of the authors' knowledge) results for mobile targets. Additionally, we provide two parametric reformulations and utilize the \ac{ipa} method~\cite{cassandras2010perturbation} in order to compute gradients and solve the parametric optimization problem locally.

The remainder of this paper is structured as follows. \secref{sec:problem:formulation} introduces the problem in detail and establishes the \ac{ocp}~\eqref{min:ocp}. In \secref{sec:optimal:control:characterization}, we show how the problem can be reduced to a finite-dimensional optimization problem. Two parametric formulations are provided and compared in \secref{sec:parameterizations}. By applying the \ac{ipa} calculus, we describe how perturbations of the parameters affect the cost function in \secref{sec:ipa}. We also compare the parameterizations in numerical experiments in \secref{sec:numerical:experiments}. \secref{sec:conclusion:future} concludes the paper and outlines ideas for future work.
\section{Problem Formulation}\label{sec:problem:formulation}
Our problem formulation builds upon the foundation of~\cite{cassandras2012optimal,zhou2018optimal}, extending it to include \emph{mobile} targets.
%The majority of this problem formulation coincides with that of~\cite{cassandras2012optimal,zhou2018optimal}, however, the extension to mobile targets is made. 
Let $\S = \R$ be the work space in which we consider the persistent monitoring problem over a finite time interval $[0,T]$. We introduce a set of mobile agents $\A = \{1, 2, \dots, N\}$, and denote the position of agent $j \in \A$ by $\agent_j(t) \in \S$. We assume that the agents follow first-order dynamics $\dot{\agent}_j(t) = u_j(t)$, where each $u_j$ belongs to the admissible control set $\U = \{ [0,T] \to [-1,1] \}$, and the initial position $s_j(0) \in \S$ is given. In addition to the agents, there is a set of mobile targets $\T = \{1, 2, \dots, M\}$, and we denote by $\target_i(t)$ and $\dtarget_i(t)$ the position and velocity of $i \in \T$ at $t \in [0,T]$. The monitoring function of agent $j$ for target $i$ is given by
\begin{equation}\label{eq:monitoring:function}
    p_{ij}(\target_i(t), \agent_j(t)) = \max\left\{0, 1 - \frac{|\target_i(t) - \agent_j(t)|}{r_j} \right\},
\end{equation}
where $r_j > 0$ is the sensing range of agent $j$. The joint monitoring function of target $i$ from all agents is then
\begin{equation}\label{eq:joint:monitoring}
    P(\target_i(t), \agent(t)) =  1 - \prod_{j \in \A}\left(1 - p_{ij}(\target_i(t), \agent_j(t)) \right).
\end{equation}

The question remains of what is being monitored. We associate each target with an internal state $R_i(t) \geq 0$ describing a measure of uncertainty. This state is assumed to grow with a constant rate of $A_i > 0$ when not monitored, and to change with a dynamic rate of $A_i - B_i P(\target_i(t), \agent(t))$ for a given $B_i > A_i$ otherwise. To ensure that the uncertainty remains nonnegative, we limit the uncertainty reduction for targets in the set
$\Z(t) = \{ i \in \T \mid R_i(t) = 0 \andt A_i < B_i P(\target_i(t), \agent(t)) \}$. Thus, the uncertainty dynamics are captured by
\begin{equation}\label{eq:uncertainty:dynamics}
    f_{R_i}(t) = \begin{cases}
        0, & \ift i \in \Z(t), \\
        A_i - B_i P(\target_i(t), \agent(t)), &\text{otherwise}.
    \end{cases}
\end{equation} 
The goal of the persistent monitoring problem is to minimize the overall uncertainty, i.e., we seek a control policy which solves the \ac{ocp}
\begin{mini!}
    {u \in \U^N}{J(u) = \frac{1}{T} \int_0^T \sum_{i \in \T} R_i(t) dt}{\label{min:ocp}}{}
    \addConstraint{\hspace*{0.1em}\dot{s}(t)}{= u(t),}{\quad \dot{R}(t) = f_R(t).}
\end{mini!}

Note that the function~\eqref{eq:uncertainty:dynamics} depends on time only through $\agent(t)$, $R(t)$, and $\target(t)$. However, to keep notation simple, we neglect the explicit dependence on the states, noting only that it is a function of time. This convention will be used throughout the paper. 

\begin{remark}
The choice of the monitoring function is not restricted to \eqref{eq:monitoring:function}. Any monitoring function can be chosen, as long as it \begin{inlineenum}
    \item has compact support, and
    \item is monotonically increasing with the distance between its arguments.
\end{inlineenum}
Similarly, the results in this paper directly extend to other joint monitoring functions, as long as they are monotonically increasing with each of the individual monitoring functions, e.g., $\max_{j \in \A} p_{ij}(\target_i(t), \agent_j(t))$.
\end{remark}

\section{Optimal Control Characterization}\label{sec:optimal:control:characterization}

As in previous work~\cite{zhou2018optimal}, a standard Hamiltonian analysis can be applied to~\eqref{min:ocp} to reveal that the optimal control can be expressed in a parametric form. To better understand the structure of the problem, however, here we show the existence of such a parametric reformulation directly. In order to prove the main result in \theoref{theo:finite-dimensional} below, we first show that any control law can be adapted to one of the form given in \eqref{eq:optimal:control} without increasing the cost. Applying this result repeatedly shows that any control can be replaced by one which can be partitioned into a finite number of intervals during which the control is either  \begin{inlineenum}
    \item constant and maximal, or
    \item equal to the velocity of a target.
\end{inlineenum} 
From this we conclude that there exists a compact parameter space, with which we can ultimately show the existence of a parametrically represented optimal control. These theoretical results rely on the following assumptions for the motions of the targets. 
\begin{assumptions}\label{assumptions:one}
    For all $t \in [0,T]$, $i,i_1 \neq i_2 \in \T$, we assume that (i) target velocities are no larger than the maximal speed of the agents, i.e., their absolute velocities are bounded by one, and (ii) agents never sense two targets at the same time, i.e., there exists  $\mindist > 0$ such that $|\target_{i_1}(t) - \target_{i_2}(t)| \geq 2\max_{j \in \A}\{r_j\} + \mindist.$
\end{assumptions}

The first assumption is fundamental, whereas the second one is of a technical nature, without which it becomes increasingly difficult to characterize the optimal control. However, in \secref{sec:numerical:experiments} we show that the considered parameterizations are adaptable to simultaneous sensing scenarios. Similarly, if a target's velocity exceeds the target control bounds, then the given parameterization with bounded controls provides a heuristic for the more general case. Our first lemma provides a characterization of the optimal control over a fixed interval. Before we begin, let us denote by `$\sign$' the sign function that maps all negative real numbers to negative one, zero to zero, and all positive real numbers to one.
\begin{lemma}\label{lem:optimal:control}
    Let $(u, \agent)$ be any feasible control-trajectory pair. Consider an interval $(t_1, t_2)$ over which the agent $j \in \A$ only ever senses one target $i \in \T$. Then, there exist $\check{t}, \hat{t} \in [t_1, t_2]$ such that the modified control law 
    \begin{equation}\label{eq:optimal:control}
        u_j^\prime(t) = \begin{cases}
            u_j(t), &\ift t \notin (t_1, t_2),\\
            \sign\left(\target_i(t_1) - \agent_j(t_1) \right), &\ift t \in (t_1, \check{t}), \\
            \dtarget_i(t), &\ift t \in (\check{t}, \hat{t}_{\phantom{}}), \\
            \sign\left(\agent_j(t_2) - \target_i(t_2)\right), &\ift t \in (\hat{t}, t_2),
        \end{cases}
    \end{equation}
    satisfies $J(u^\prime) \leq J(u)$. 
\end{lemma}
\begin{proof}
    In what follows we construct a trajectory $\agent_j^\prime$ which coincides with that of $\agent_j$ everywhere outside of $(t_1,t_2)$, and is at least as close to target $i$ as $\agent_j$ during $(t_1, t_2)$. Let us begin by capturing the signs in~\eqref{eq:optimal:control} by $\gamma_1 = \sign(\target_i(t_1) - \agent_j(t_1))$ and $\gamma_2 = \sign(\agent_j(t_2) - \target_i(t_2))$, which are constant throughout this proof. We then introduce the auxiliary times
    \begin{equation*}
        \begin{split}
            \tilde{t}_1 &= \inf\{ t > t_1 \mid \agent_j(t_1) + \gamma_1(t - t_1) = \target_i(t) \}, \\
            \tilde{t}_2 &= \sup\{ t < t_2 \mid \target_i(t) + \gamma_2(t_2 - t) = \agent_j(t_2) \}.
        \end{split}
    \end{equation*}    
    Geometrically speaking, $\tilde{t}_1$ is the earliest arrival time at the target (possibly infinity), and $\tilde{t}_2$ is the latest departure time from the target such that the agent still reaches the point $\agent_j(t_2)$ (possibly negative infinity).
    If $\tilde{t}_1 \leq \tilde{t}_2$, then we set $\check{t} = \tilde{t}_1$ and $\hat{t} = \tilde{t}_2$ (see~\figref{fig:proof:sketch}). Otherwise, the agent does not reach the target within the given interval, which implies $\gamma_1 = -\gamma_2$, and we find $\check{t} = \hat{t}$ by solving $\agent_j(t_1) + \gamma_1(\check{t} - t) - \gamma_1(t_2 - \check{t}) = \agent_j(t_2)$. 
    The geometric interpretation of this is that the agent moves towards the target at full speed until it must return at full speed in order to end up at $\agent_j(t_2)$.
    Solving for $\check{t}$ yields $\check{t} = \frac{1}{2}\left( \gamma_1(\agent_j(t_2) - \agent_j(t_1)) + (t_1 + t_2) \right)$. Note that $\check{t} \in [t_1, t_2]$, since $| \agent_j(t_2) - \agent_j(t_1)| \leq t_2 - t_1$. 
    
    Now let $\agent_j^\prime$ be the trajectory obtained from $u^\prime$. By construction, this trajectory coincides with that of $\agent_j$ everywhere outside of $(t_1, t_2)$, and all other agents remain unchanged. Additionally, we have $|\target_i(t) - \agent_j^\prime(t)| \leq |\target_i(t) - \agent_j(t)|$ for all $t \in (t_1, t_2)$. Due to the monotonicity of the monitoring functions, 
    %Conequently, \remove{$p_{ij}(\target_i(t), \agent_j^\prime(t)) \geq p_{ij}(\target_i(t), \agent_j(t))$. Since the joint monitoring function is monotonically increasing in each monitoring function, and $i$ is the only sensed target, it follows that $P(\target_i(t), \agent^\prime(t)) \geq P(\target_i(t), \agent(t))$. Thus, from (\ref{eq:uncertainty:dynamics}), $\dot{R}_i(t, \agent^\prime) \leq \dot{R}_i(t, \agent)$, and finally 
    % \begin{equation*}
    %     R_i(t, \agent_j^\prime) - R_i(t, \agent_j) = \int_{t_1}^{t} [\dot{R}_i(\sigma, \agent^\prime) - \dot{R}_i(\sigma, \agent)]  d\sigma \leq 0
    % \end{equation*}    
    % for all $t \in (t_1, t_2)$. This shows that the cost ...} 
    the cost under $u^\prime$ can be no larger than that under $u$.
\end{proof}

\begin{figure}
    \centering
    \includegraphics[width=\figwidth]{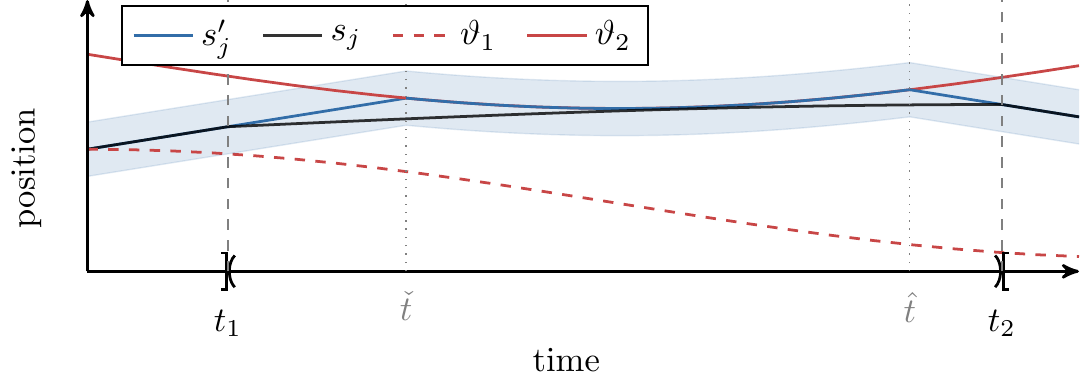}
    \caption{Visualization of the trajectory obtained from control law~\eqref{eq:optimal:control}. Note how the agent only senses the second target over the interval $(t_1, t_2)$. In view of the proof of~\lemref{lem:optimal:control}, the times $\check{t}$ and $\hat{t}$ are equal to the auxiliary times $\tilde{t}_1$ and $\tilde{t}_2$ as tracking is achieved.}
    \label{fig:proof:sketch}
\end{figure}

This lemma allows us to decompose an existing control into sequential phases (or modes) in which only one target is visited. We utilize this control modification for each individual interval to obtain a parameterized control law. For the next lemma we introduce the ceiling function $\lceil \cdot \rceil$, which maps each real number to the smallest integer greater than or equal to the number itself.

\begin{lemma}\label{lem:finite:parameters}
    Let $(u,s)$ be any feasible control-trajectory pair. Then there exists an alternative control law $u^\prime$, which satisfies $J(u^\prime) \leq J(u)$ and is fully described by $n = 5 N \lceil T/\mindist \rceil$ parameters.
\end{lemma}
\begin{proof}
    We begin with $u_j^\prime = u_j$ and denote by $(i_1, i_2, \dots, i_k)$ %\in \T^k$ 
    the temporally ordered sequence of the sensed targets along the trajectory $s_j$ satisfying $i_\ell \neq i_{\ell+1}$ for all $\ell$. Note that the length of this sequence $k$ is bounded by $n =  \lceil T/\mindist \rceil$ due to \assumptionsref{assumptions:one}, and without loss of generality we assume $k \geq 1$.
    
    Similar to the target sensing sequence, we seek to find switching times $0 \leq \bar{t}_1 \leq \dots \leq \bar{t}_{k} = T$, and thereby switching positions $s_j(\bar{t}_1), s_j(\bar{t}_2), \dots, s_j(\bar{t}_k)$, such that $\bar{t}_\ell$ describes the end of the sensing phase of target $i_\ell$. More precisely, $\bar{t}_\ell$ is defined as follows. Setting $\bar{t}_0 = 0$, then we define for $1 \leq \ell \leq k-1$ the upper bound 
    \begin{equation}
        \bar{t}_{\ell}^\mathrm{max} = \inf \{ t > \bar{t}_{\ell-1} \mid |\target_{i_{\ell+1}}(t) - \agent_j(t) | = r_j\},
    \end{equation}
    which is the earliest time of sensing the next target, and with that
    \begin{equation}\label{eq:partition:times}
        \bar{t}_{\ell} = \sup \{ t < \bar{t}_{\ell}^\mathrm{max} \mid |\target_{i_{\ell}} - \agent_j(t)| = r_j \}.
    \end{equation}
    By construction, agent $j$ only senses target $i_\ell$ along its trajectory $s_j\vert_{(\bar{t}_{\ell-1}, \bar{t}_{\ell})}$ restricted to $(\bar{t}_{\ell-1}, \bar{t}_\ell)$. According to~\lemref{lem:optimal:control}, replacing $u_j^\prime\vert_{(\bar{t}_{\ell-1}, \bar{t}_{\ell})}$ by~\eqref{eq:optimal:control} does not increase the cost. This replacement yields two more parameters $\check{t}_\ell, \hat{t}_\ell$, which describe the switching times of the modes.
    
    Note that the new control law $u^\prime_j\vert_{(\bar{t}_{\ell-1}, \bar{t}_{\ell})}$ is now fully described by the switching times $\bar{t}_{\ell-1}$, $\check{t}_\ell$, $\hat{t}_\ell$, and $\bar{t}_\ell$, the sensed target $i_\ell$, and the switching positions $s_j(\bar{t}_{\ell-1}), s_j(\bar{t}_\ell)$. 
    
    Repeating this procedure for each sensed target yields a control described by $5k$ parameters ($k$ target indices, $3k$ switching times, and $k$ switching positions). In order to unify the dimensions across the agents, let us embed the parameters into the respective $n$-dimensional spaces by replicating the last parameters. This procedure is repeated for all agents, resulting in a total number of $5Nn$ parameters.      
\end{proof}

From this property, we can immediately conclude that if the \ac{ocp}~\eqref{min:ocp} admits a solution, then there exists a control described by a finite number of parameters. In order to show later that such a solution always exists, we now prove that the resulting parameter space is compact.

\begin{lemma}\label{lem:compact:parameter:space}
    There exists a finite-dimensional and compact parameter space $\Y$ with functions 
    \begin{equation}\label{eq:param:space:functions}
        T : \U^N \to \Y \quad \andt \quad H : \Y \to \U^N,
    \end{equation}
    such that $J(H(T(u))) \leq J(u)$. In addition, the function $H$ can be chosen to be continuous.
\end{lemma}
\begin{proof}
    Consider $\tilde{\Y} = \T^n \times \R^{n} \times \R^{n} \times \R^{n} \times \R^{n}$, and $y = (i, \bar{t}, \check{t}, \hat{t}, s) \in \tilde{\Y}$. Then, define 
    \begin{equation*}
        \Y_j = \left\{y \in \tilde{\Y} \mid \bar{t}_{\ell-1}\leq \check{t}_\ell \leq \hat{t}_\ell \leq \bar{t}_\ell, |s_{\ell-1}-s_0|\leq \bar{t}_{\ell-1}\forall \ell\right\},
    \end{equation*}
    where $\bar{t}_0 = 0$ and $\bar{t}_n = T$. The dependence on $j \in \A$ only affects $s_0 = s_j(0)$, which is given. The full parameter space is then defined by $\Y = \Y_1 \times \Y_2 \times \dots \times \Y_N$. Note that each $\Y_j$ is bounded and closed, and $\Y$ is thus compact. The function $T : \U^N \to \Y$ is defined by the mapping constructed in the proof of~\lemref{lem:finite:parameters}.

    On the other hand, the function $H : \Y \to \U^N$ simply maps an element $y \in \Y$ to the control $u^\prime(t)$, where 
    \begin{equation}\label{eq:optimal:control:interval}
        u_j^\prime(t) = \begin{cases}
            \sign\left(\target_{i_\ell^j}(\bar{t}^j_{\ell-1}) - s^j_{\ell-1}\right), & t \in [\bar{t}_{\ell-1}^j, \check{t}_\ell^j), \\
            \dtarget_{i_\ell^j}(t), &t \in [\check{t}_{\ell}^j, \hat{t}_\ell^j), \\
            \sign\left(s^j_{\ell} - \target_{i_\ell^j}(\bar{t}^j_{\ell})\right), & t \in [\hat{t}_{\ell}^j, \bar{t}_\ell^j).
        \end{cases}
    \end{equation}
    Then $H(T(u)) = u^\prime$ as in~\lemref{lem:finite:parameters}, and consequently $J(H(T(u))) \leq J(u)$. 

    Finally, we show that $H$ is continuous. Let $\pi_m \to \pi$ be any convergent sequence in the parameter space, and let $v_m = H(\pi_m)$ and $v = H(\pi)$. Let $\varepsilon > 0$ be arbitrarily small and define $\delta = \min\{\frac{\varepsilon}{13Nn}, \frac{\tilde{\delta}}{3}\}$, where $\tilde{\delta} > 0$ is the smallest positive distance between the target positions and switching points that determine the signs in~\eqref{eq:optimal:control:interval} for the parameterization $\pi$. Then, due to the convergence in the parameter space, there exists a $K>0$ such that the absolute value of each component of $\pi_m - \pi$ is bounded by $\delta$ for all $m \geq K$. By the choice of $\delta$, the evaluated signs in~\eqref{eq:optimal:control:interval} agree for all parameterizations $\pi_m$ with $m \geq K$. Now let $(\bar{t}^j_{\ell-1}, \bar{t}^j_{\ell})$ be an interval between switching points given by the parameterization $\pi$. During this interval the control laws $v_m = H(\pi_m)$ and $v = H(\pi)$ can only differ over six intervals of length $\delta$; one each for the beginning and the end of the intervals in~\eqref{eq:optimal:control:interval}. Over those intervals, the control deviation is bounded by $2$. In total we get
    \begin{equation}
        \begin{split}
            \|v_m - v \| &= \sum_{j=1}^N \sum_{\ell = 1}^n \int_{\bar{t}^j_{\ell-1}}^{\bar{t}^j_\ell} |v_{m,j}(t) - v_{j}(t)| dt \\
            &\leq N n 6 2 \delta = 12Nn\delta < \varepsilon.
        \end{split}
    \end{equation}
    This shows that $H$ is continuous.
\end{proof}

The fact that the parameter space is compact leads to the existence of converging subsequences, which we utilize in the next theorem.

\begin{theorem}\label{theo:finite-dimensional}
    Consider the~\ac{ocp}~\eqref{min:ocp}. Under~\assumptionsref{assumptions:one} there exists an optimal control $u^\ast \in \U^N$ and an optimal parameterization $\pi^\ast \in \Y$ such that $H(\pi^\ast) = u^\ast$.
\end{theorem}
\begin{proof}
    Note that the cost function $J$ is bounded. Thus, there exists an infimal cost $J^\ast$ and, consequently, a sequence $(u_k)_{k \in \mathbb{N}}$ of controls such that $J(u_k) \to J^\ast$. Denote by $T(u_k) = \pi_k$ the mapped elements of the control sequence into the parameter space $\Y$. Since $\Y$ is compact, the mapped sequence must admit a convergent subsequence $\lim\limits_{m \to \infty} \pi_{k_m} = \pi^\ast \in \Y$. The continuity of $H$ provides 
    \begin{equation*}
        \lim_{m \to \infty} H(\pi_{k_m}) = H\left(\lim_{m \to \infty} \pi_{k_m}\right) = H(\pi^\ast) = u^\ast \in \U^N, 
    \end{equation*}
    which ultimately shows 
    \begin{equation*}
        J(u^\ast) = J\left( \lim_{m \to \infty} H(\pi_{k_m})\right) = \lim_{m \to \infty} J\bigl( H(\pi_{k_m})\bigr) = J^\ast
    \end{equation*}
    due to the continuity of the cost function.
\end{proof}

\begin{figure}
    \centering
    \includegraphics[width=\figwidth]{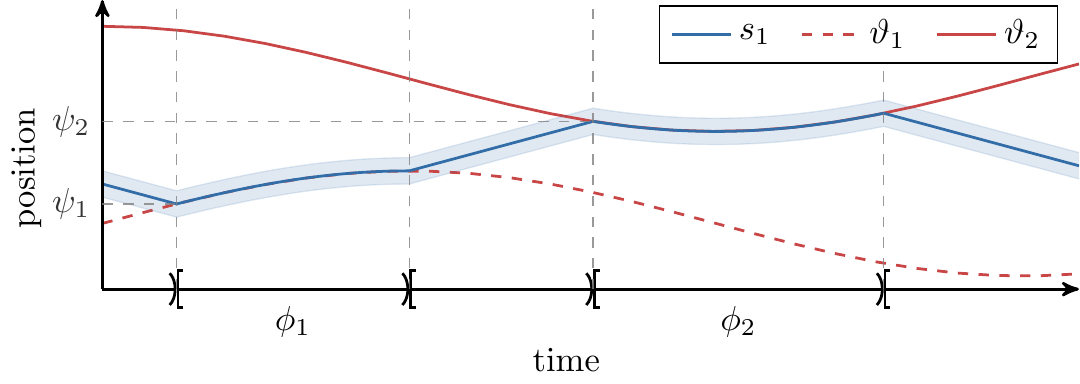}
    \includegraphics[width=\figwidth]{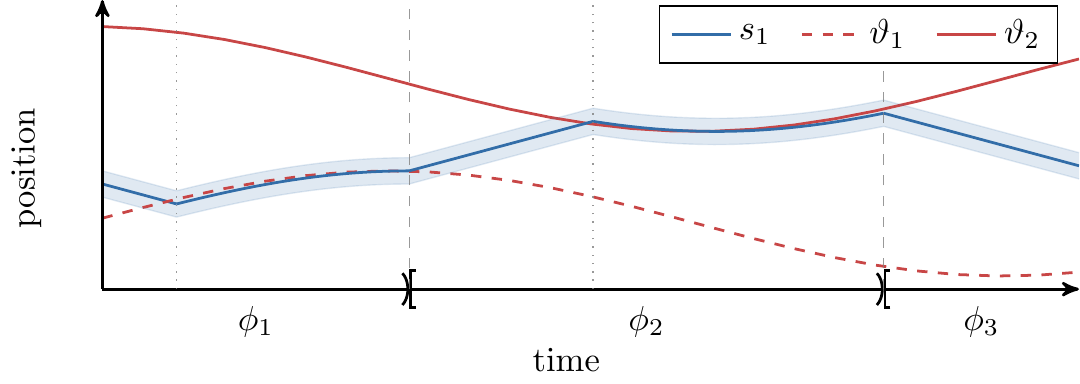}
    \caption{Comparison of the optimal (top) and practical (bottom) parameterizations. The variables below the intervals indicate their length and type, according to the parameterization. Dashed vertical lines represent foreseeable mode changes due to proceeding to the next interval, whereas dotted vertical lines depict observable mode changes due to sufficient reduction of the tracking error.}
    \label{fig:parameterizations}
\end{figure}

\section{Parametric Reformulation}\label{sec:parameterizations}
The above analysis shows that the optimal control problem can be reduced to a much simpler \emph{parametric} optimization problem. In what follows, we will introduce two different parameterizations. The first one provides a generalization of~\eqref{eq:optimal:control:interval}, and we refer to it as the \textit{optimal parameterization}, since it contains an optimal control as established in \theoref{theo:finite-dimensional}. Its drawback is that it requires knowledge of the target velocities, which are difficult to estimate in practice. From a practical viewpoint, the policy~\eqref{eq:optimal:control:interval} can be read as \textit{stay as close to the target as possible}, which motivates the alternative  \textit{practical parameterization}. Though it loses optimality to some degree, it only requires positional feedback of the targets, and has additional advantages, such as a natural robustness towards noise, and simplicity of initialization; details are discussed later on. We now introduce both forms for a fixed agent $j$, while omitting the index for notational simplicity. For guidance, we refer the reader to~\figref{fig:parameterizations}, which illustrates both formulations.

\noindent\textbf{Optimal Parameterization}
Let us first define the set of convex combinations $\C = \{\trackcomb \in [0,1]^M \mid \sum_{i=1}^M\trackcomb_i = 1 \}$. Next, we introduce parameters that can be divided into switching points $\switchpos_1,\dots,\switchpos_L \in \S$, tracking combinations $\trackcomb_1, \dots, \trackcomb_L \in \C$ and tracking durations $\trackinter_1, \dots, \trackinter_L \in [0,T]$, where $L$ is some fixed integer. From this parameterization we obtain a control 
\begin{equation}\label{eq:optimal:param:control}
    u(t) = \begin{cases}
        \sign(\switchpos_\ell - s(\bar{t}_{\ell-1})), &\fort t \in [\bar{t}_{\ell-1}, \check{t}_\ell), \\
        \trackcomb_\ell^\top \dtarget(t), &\fort t \in [\check{t}_{\ell}, \bar{t}_{\ell}),
    \end{cases}
\end{equation}
where $\check{t}_\ell = \bar{t}_{\ell-1} + |\switchpos_\ell - s(\bar{t}_{\ell-1})|$ and $\bar{t}_\ell = \check{t}_\ell + \trackinter_\ell$, initialized with $\bar{t}_0 = 0$.

There are two reasons why we propose the target tracking to be encoded in the form of convex combinations. First and foremost, the feasible set of the parametric description becomes convex. The only difficulty to overcome in an optimization scheme is that of the cost function's nonconvexity and nonsmoothness. Secondly, the resulting control law has some adaptability towards simultaneous sensing scenarios, as discussed in~\secref{sec:numerical:experiments}. Note that if \assumptionsref{assumptions:one} is satisfied, then the optimal control established in \theoref{theo:finite-dimensional} is contained in the parameterization by choosing convex combinations that only select a single target at a time.

\noindent\textbf{Practical Parameterization.} As opposed to defining switching points in the mission space and then matching a convex combination of target velocities, we now propose a control law by directly tracking a convex combination of target positions. 

The parameterization is thus reduced to tracking durations $\trackinter_1, \dots, \trackinter_L \in [0,T]$ and tracking combinations $\trackcomb_1,\dots,\trackcomb_L \in \C$. The switching times are simply given by $\bar{t}_\ell = \bar{t}_{\ell-1} + \trackinter_\ell$, which is again initialized with $\bar{t}_0 = 0$. For $t \in [\bar{t}_{\ell-1}, \bar{t}_{\ell})$ we then obtain a control
\begin{equation}\label{eq:PI:controller}
    u(t) = \min\left( 1, \max\left(-1,  f^\mathrm{PI}_\ell(t) \right)\right),
\end{equation}
via the PI-controller
\begin{equation}
    f^\mathrm{PI}_\ell(t) = \Kp e_\ell(t) + \Ki \int_{\tilde{t}_\ell}^{\max(t, \tilde{t}_\ell)} e_\ell(\sigma) d\sigma,
\end{equation}
with tracking error $e_\ell(t) = \trackcomb_\ell^\top \target(t) - \agent(t)$, and respective proportional and integral feedback gains $\Kp > 0$ and $\Ki > 0$. Since the integral part is useful for driving a small proportional error to zero, we activate this part once the tracking error becomes small enough. Thus, the starting time of the integrator part can be written as
\begin{equation}
    \tilde{t}_\ell = \inf\{ t \geq \bar{t}_{\ell-1} \mid |\Kp e_\ell(t)| \leq \varepsilon_\mathrm{tol} \},
\end{equation}
for some desired switching tolerance $\varepsilon_\mathrm{tol} \in (0,1)$. The motivation behind this strategy is that the beginning of a new tracking interval typically leads to a large proportional feedback gain due to the switched tracking combination $\trackcomb_\ell$. In order to prevent integrator windup and consequent control saturation during the PI tracking phase, under this strategy we first track via the P-controller, switching to a PI-controller once the tracking error becomes small enough. Note that this is a fairly simple anti-windup scheme, with the main goal of keeping the control law continuous and thereby the IPA analysis as tractable as possible; in practice more standard anti-windup algorithms may be preferred.

\noindent\textbf{Comparison of the parameterizations.} The practical parameterization clearly loses optimality, the degree of which depends on the position sampling rate and the chosen feedback gains. On the other hand, the practical implementation gives rise to attractive properties. One is the ease of initialization. Assume that we want to initialize the system with a given visiting schedule of targets. Then, all we have to do is decide how long we want to track each of the targets. Trying to initialize the optimal parameterization for mobile targets becomes tedious and complex; it requires estimates of where to intercept the targets and where agents will end up after each tracking phase. Such information is rarely available in practice. Additionally, the practical implementation is naturally more robust to perturbations due to its closed loop design.

% \begin{figure}
%     \centering
%     \includegraphics[width=0.5\linewidth]{figures/suboptimality/fig.pdf}
%     \caption{Loss of optimality of parameterization II ($\agent_2$, red, dashed) visualized for a proportional feedback driven controller in discrete time. The parameterization I ($\agent_1$, green, solid) is able to switch its mode once it matches the targets position ($\target$, blue, solid), as it is event based.}
%     \label{fig:suboptimality}
% \end{figure}
\section{Optimizing the System Parameters}\label{sec:ipa}
The parametric controllers introduced in the preceding section define a hybrid system for each agent. We denote the feasible set of parameters as $\Theta$, the precise form of which depends on the chosen parameterization. The \ac{ocp}~\eqref{min:ocp} is now reduced to the parametric optimization problem 
\[ \minimize_{\allparams \in \Theta}~J(\allparams). \]
In order to perform a gradient descent method we need the gradient of $J$ with respect to $\theta$; we find this using \ac{ipa} \cite{cassandras2010perturbation}. It was shown in~\cite{cassandras2012optimal, zhou2018optimal} that the derivative of the cost function with respect to the parameters is given by
\begin{equation*}
    \frac{\partial J(\mathbf{\allparams})}{\partial \allparams} = \frac{1}{T} \sum_{k=1}^{K} \sum_{i = 1}^M \int_{\eventtime_{k-1}(\mathbf{\allparams})}^{\eventtime_{k}(\mathbf{\allparams})} \frac{\partial R_i(\sigma)}{\partial \allparams} d\sigma,
\end{equation*}
where $\eventtime_0, \eventtime_1, \dots, \eventtime_{K}$ describe the event times when mode switches occur in the hybrid system, and $K$ denotes the total number of events (note $\eventtime_0 = 0$ and $\eventtime_{K} = T$). We now denote by $x^\prime(t) = \frac{\partial x(\allparams, t)}{\partial \allparams}$ the Jacobian matrix, where $x = (\agent, R)$ is the collection of states. We point out that the target positions and velocities could be included as states, however, since they do not depend on the system parameters they can be omitted. We now compute the gradients of $R_i$ by applying the \ac{ipa} equation~\cite{cassandras2010perturbation}
\begin{equation}\label{eq:ipa:1}
    \frac{d}{dt} x^\prime(t) =  \frac{\partial f_k(t)}{\partial x} x^\prime(t) + \frac{\partial f_k(t)}{\partial \theta}
\end{equation}
over an inter-event interval $[\eventtime_k, \eventtime_{k+1})$ with boundary condition
\begin{equation}\label{eq:IPA:2}
    x^\prime(\eventtime_k^+) = x^\prime(\eventtime_k^-) + \left( f_{k-1}(\eventtime_k^-) - f_k(\eventtime_k^+) \right) \eventtime_k^\prime.
\end{equation}
Clearly, $R_i^\prime(t) = 0$ if $i \in \Z(t)$. Otherwise, we use~\eqref{eq:ipa:1} to obtain
\begin{equation}\label{eq:nabla_R}
    \frac{d}{dt} \frac{\partial R_i(t)}{\partial \allparams_j}
     = -B_i \frac{\partial p_{ij}(t)}{\partial \agent_j} \frac{\partial \agent_j(t)}{\partial \allparams_j} \prod\limits_{\ell \neq j} (1 - p_{i\ell}(t)).
\end{equation}
The gradient of $R$ may experience discontinuities. This can only be the case when events occur that change the dynamics of $R$. As previously shown in~\cite{cassandras2012optimal},
\begin{equation}
    \frac{\partial R_i(\eventtime_k^+)}{\partial \param_j} = \begin{cases}
        \frac{\partial R_i(\eventtime_k^-)}{\partial \param_j}, &\ift R_i(\eventtime_k) \neq 0, \\
        0, &\ift R_i(\eventtime_k) = 0.
    \end{cases}
\end{equation}
The remaining unknown term in~\eqref{eq:nabla_R} is the gradient of $s$ with respect to the system parameters. Applying~\eqref{eq:ipa:1} once again results in 
\begin{equation}\label{eq:generic:IPA}
    \frac{d}{dt} \agent^\prime(t) = \frac{\partial f_k(t)}{\partial s}\agent^\prime(t) + \frac{\partial f_k(t)}{\partial \theta},
\end{equation}
with boundary condition 
\begin{equation}\label{eq:IPA:boundary}
    \agent^\prime(\tau_k^+) = \agent^\prime(\tau_k^-) + \left(f_{k-1}(\tau_{k}^-) - f_{k}(\tau_k^+) \right) \tau_k^\prime.
\end{equation}
Note that the analysis up to this point was independent of the control parameterization. We now compute the specific forms of~\eqref{eq:generic:IPA} and~\eqref{eq:IPA:boundary} separately for the individual parameterizations. In both cases we fix an index $j \in \A$ and refrain from explicitly writing it out.

\noindent\textbf{IPA for the optimal parameterization.} We recall from~\eqref{eq:optimal:param:control} that the optimal control is described by two modes: \begin{inlineenum}
    \item the switching mode $f_\ell^s(t) = \sign(\switchpos_\ell - s(t_{\ell-1}))$; and
    \item the tracking mode $f_\ell^t(t) = \trackcomb_\ell^\top \dtarget(t)$.
\end{inlineenum} Let $t \in [\tau_k, \tau_{k+1})$, and assume first that the current mode is the $\ell$th tracking mode, i.e., $f_k(t) = \trackcomb_\ell^\top \dtarget(t)$. Then the r.h.s. of~\eqref{eq:generic:IPA} vanishes, with the exception being the entries $
    \frac{d}{dt} \frac{\partial \agent(t)}{\partial \alpha_{\ell,i}} = \dtarget_i(t)$ for all $i \in \T$. Otherwise, if the current mode is $f_\ell^s(t)$, then $\agent^\prime$ is constant. We are left with having to identify the time derivatives of the boundary condition~\eqref{eq:IPA:boundary}. These values only have to be computed for mode switches that change the dynamics of the agent, since otherwise $f_{k-1} = f_k$ is a continuous transition between the modes. There are precisely two cases that cause such switches: ones that are caused by reaching a switching point and ones that are caused by finishing a tracking period.

If the switch is triggered by reaching the switching point $\switchpos_\ell$, then the event is endogenous and triggered by the switching function, which we define by $g_\ell(\agent(\allparams, t), \allparams) = \switchpos_\ell - \agent(t)$, becoming zero. Then, as shown in~\cite{cassandras2010perturbation}, we find
\begin{equation*}
    \begin{split}
        \tau_k^\prime &= -\left( \frac{\partial g_k}{\partial s} f_{k-1}(\tau_k^-)\right)^{-1} \left( \frac{\partial g_k}{\partial \allparams} + \frac{\partial g_k}{\partial s} x^\prime(\tau_k^-)\right) \\
        &= \sign(\switchpos_\ell - \agent(\bar{t}_{\ell-1})) \left( \frac{\partial g_k}{\partial \allparams} - x^\prime(\tau_k^-)\right),
    \end{split}
\end{equation*}
where $\frac{\partial g_k}{\partial \allparams}$ vanishes for all parameters except for $\frac{\partial g_k}{\partial \switchpos_\ell} = 1$.

On the other hand, if the switch is caused by leaving a tracking period, then the switch is an induced event triggered at time $\tau_k = \tau_m + \phi_\ell$, where $\tau_m$ is the most recent switching time of the agent's mode. From this we find that $\tau_k^\prime$ evolves continuously except for the parameter $\trackinter_\ell$, for which the new partial is given by $\frac{\partial \tau_k}{\partial \trackinter_\ell} = 1$, since this parameter did not affect any previous event times.

\noindent\textbf{IPA for the practical parameterization.} The practical parameterization implicitly defines four modes: 
\begin{inlineenum} \item $f_\ell^1$ for constant positive bang control; \item $f_\ell^{-1}$ for constant negative bang control; \item $f_\ell^\mathrm{P}$ for the purely proportional tracking control while respecting the control bounds; and \item $f_\ell^\mathrm{PI}$ for a PI tracking control respecting the control bounds. \end{inlineenum} Then, as before, we find that $s^\prime$ is constant during the modes $f_\ell^1$ and $f_\ell^{-1}$. Now let us assume that $f_\ell^\mathrm{P}$ is active. Then, we find $\frac{\partial f_k(t)}{\partial s} = -\Kp$, whereas $\frac{\partial f_k(t)}{\partial \allparams} = 0$ for all parameters, except $\frac{\partial f_k(t)}{\partial \trackcomb_{\ell,i}} = \Kp\target_i(t)$. And finally, if the PI-control law $f_\ell^\mathrm{PI}$ is active, then $\frac{\partial f_k(t)}{\partial s} = -(\Kp + \Ki(t - \tilde{t}_\ell))$, and again $\frac{\partial f_k(t)}{\partial \allparams} = 0$ for all parameters except 
\begin{equation*}
\frac{\partial f_k(t)}{\partial \alpha_{\ell,i}} = \Kp\target_i(t) + \Ki \int_{\tilde{t}_\ell}^{t} \target_i(\sigma) d\sigma.
\end{equation*}

Again, the state derivatives $\agent^\prime$ may experience discontinuities due to~\eqref{eq:IPA:boundary}. This only occurs when an agent's mode switches. Recall that the controller~\eqref{eq:PI:controller} was designed to provide a continuous function during the interval $(\bar{t}_{\ell-1}, \bar{t}_\ell)$, thus switches within the modes of~\eqref{eq:PI:controller} do not lead to discontinuities in the state derivatives. On the other hand, if the switch is caused by leaving the $\ell$th tracking period, then the switch is an induced event triggered at time $\eventtime_k = \trackinter_0 + \trackinter_1 + \dots + \trackinter_\ell$. We can directly compute $\frac{\partial \tau_k}{\partial \trackcomb} = 0$ and $\frac{\partial \tau_k}{\partial \trackinter_m} = 1$ for all $m \leq \ell$.

\noindent\textbf{Finding a feasible search direction.}
With all gradients determined via IPA, we are almost ready to perform a gradient descent method. However, in order for the tracking combination constraint to remain satisfied, the step direction with respect to the tracking constraints, denoted by $p_\trackcomb$, must be chosen appropriately, i.e., its components must add up to zero for each tracking period $\ell$ and agent $j$. We do so by computing a feasible step direction via the least squares problem
\begin{mini!}
    {\stepdirection_\alpha}{\| \frac{\partial J}{\partial \alpha} - \stepdirection_\alpha \|_2^2}{}{\label{min:LP}}
    \addConstraint{0 \leq \alpha_{\ell i j} + \stepdirection_{\alpha_{\ell i j}}}{\leq 1~}{\forall \ell, i, j,}
    \addConstraint{\sum_{m=1}^M \stepdirection_{\alpha_{\ell m j}}}{ = 0}{~\forall \ell,j,}
\end{mini!}
where $\stepdirection_{\alpha_{\ell i j}}$ is the $i$th component of the step direction for agent $j$ in the $l$th tracking period. The global solution of this convex optimization problem yields the steepest gradient of $J$ respecting the tracking parameter constraints. For the remaining parameters, we directly use the IPA gradients. In order to achieve cost function descent at each iterate, we choose a step length via the Armijo backtracking technique~\cite{armijo1966minimization}.

\begin{figure}
    \centering
    \includegraphics[width=\figwidth]{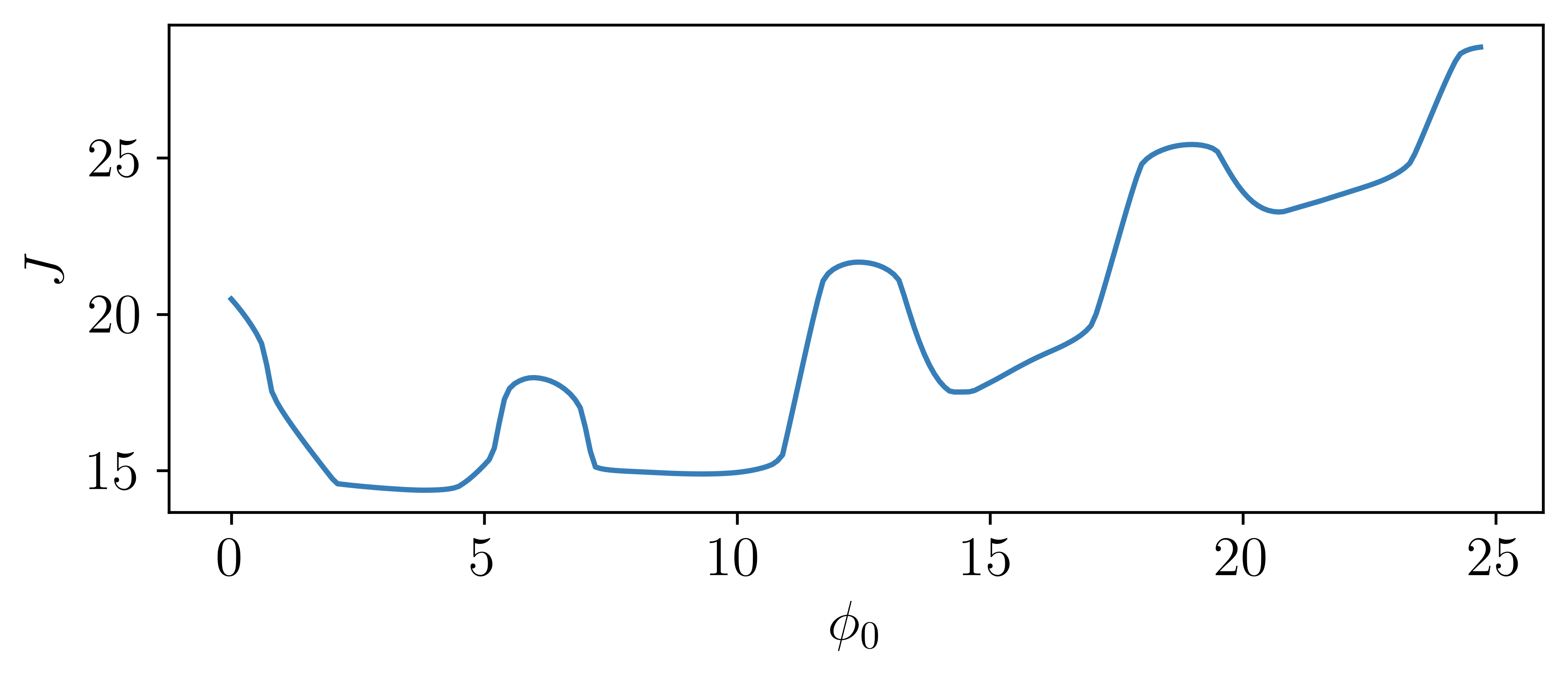}
    \includegraphics[width=\figwidth]{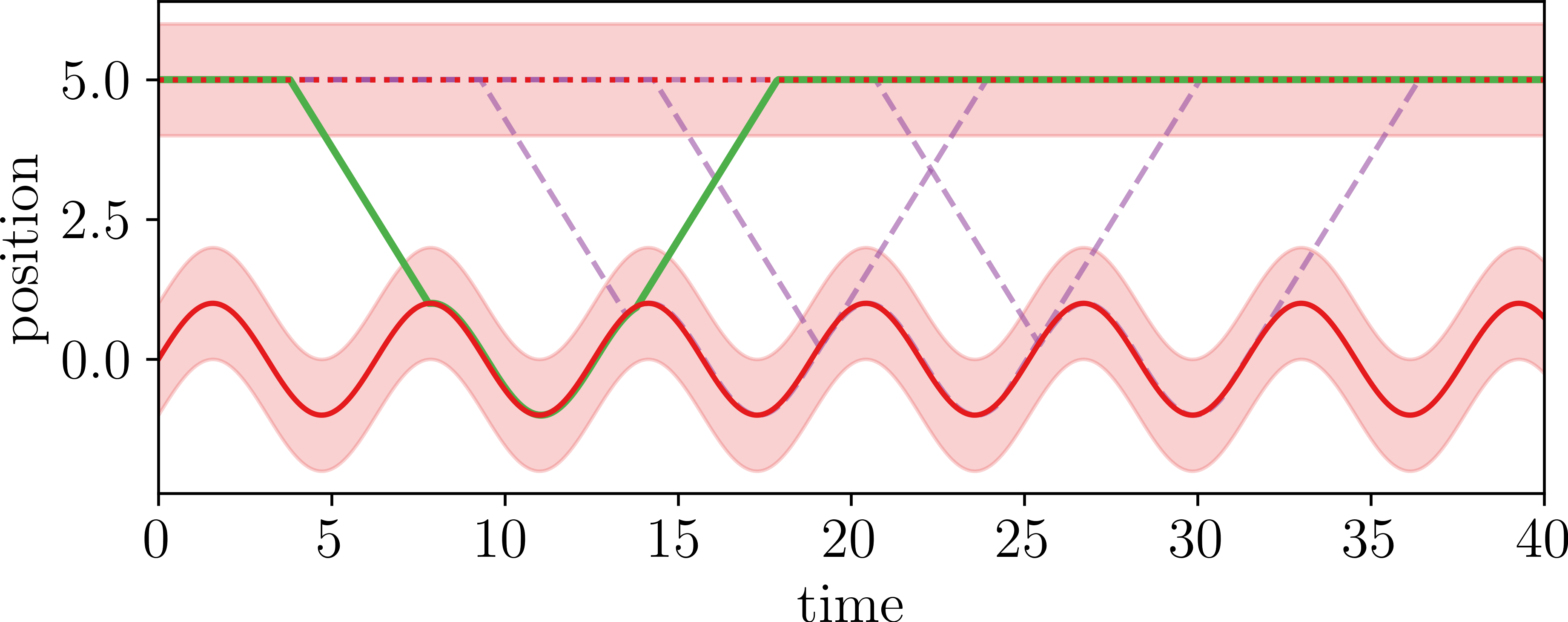}
    \caption{Depiction of the highly nonconvex cost (top) together with respective local (purple) and global (green) solutions (bottom) when fixing all but one parameter. Note that the sensing range (red) is shown around the targets as opposed to around the agents.}
    \label{fig:local:optimality:intervals}
\end{figure}
\noindent\textbf{Challenges in finding good local solutions}~\label{sec:local:solutions}
Though we are now able to optimize the system parameters over a convex set, and gradients are available for suitable initializations, the problem remains difficult to solve due to the local nature of the gradient descent procedure. For the one-dimensional persistent monitoring problem for static targets, it was demonstrated that finding the global solution of a problem with fixed target visiting sequence is tractable~\cite{zhou2018optimal}.~\figref{fig:local:optimality:intervals} demonstrates that this changes drastically in the mobile target case. The cost, even with respect to a single parameter, can be highly nonconvex. Moreover, the cost function is in parts extremely steep, whereas for other segments it is very flat. This combination presents a challenge for gradient descent methods and highlights the need to select appropriate step sizes in order to avoid stagnation and remain in the contraction area of a desirable solution.

Another difficulty, which is typical for event based algorithms, is the potential lack of event excitation~\cite{khazaeni2016event}. This occurs in the persistent monitoring problem if the agents change their modes while not sensing targets. In such cases the \ac{ipa} gradients fail to estimate the true cost gradient resulting in a failure of the optimization scheme. Though beyond the scope of this paper, the methods discussed in~\cite{khazaeni2016event} could help overcome this problem.

\begin{figure}
    \centering
    \includegraphics[width=\columnwidth]{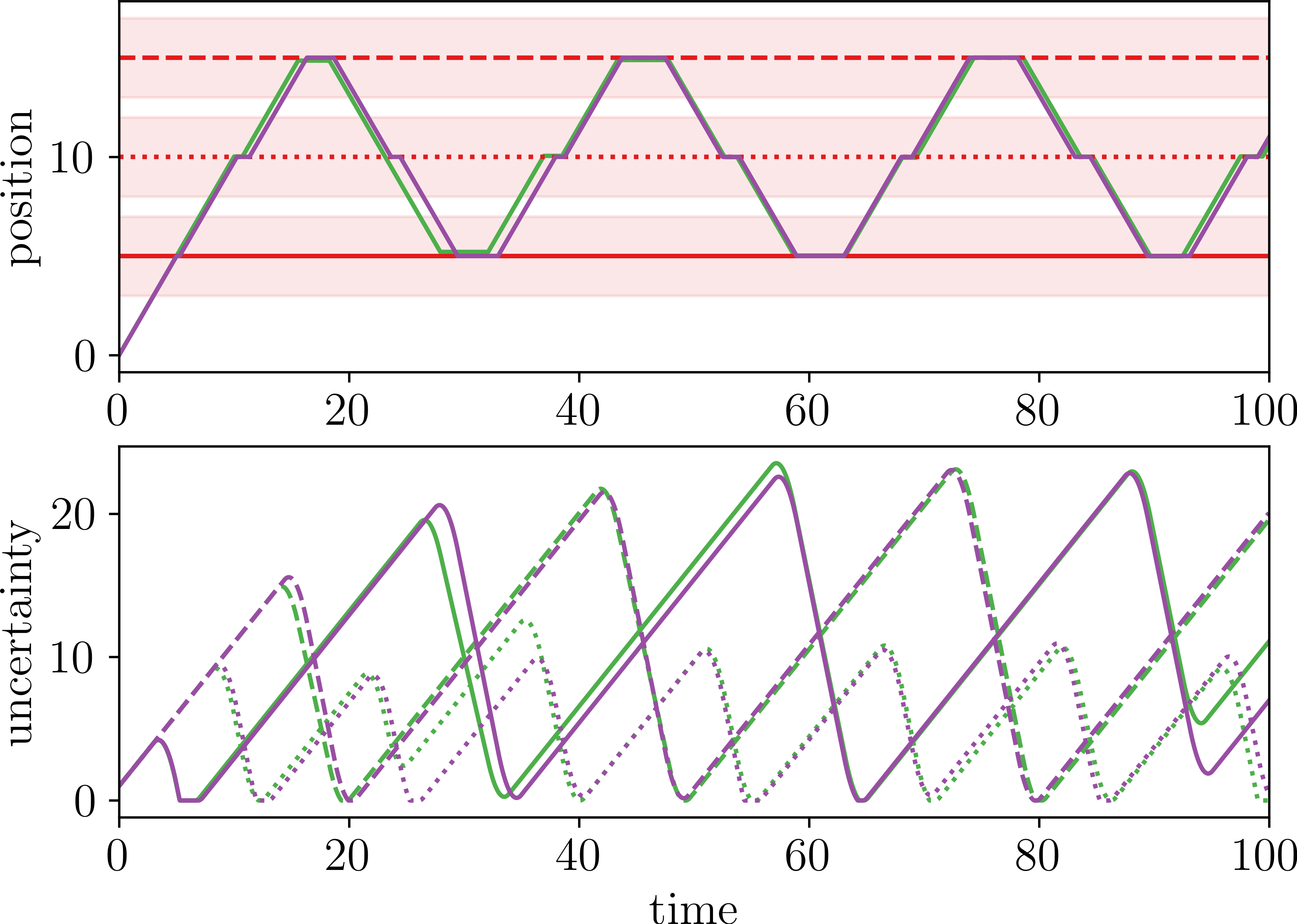}
    \includegraphics[width=\columnwidth]{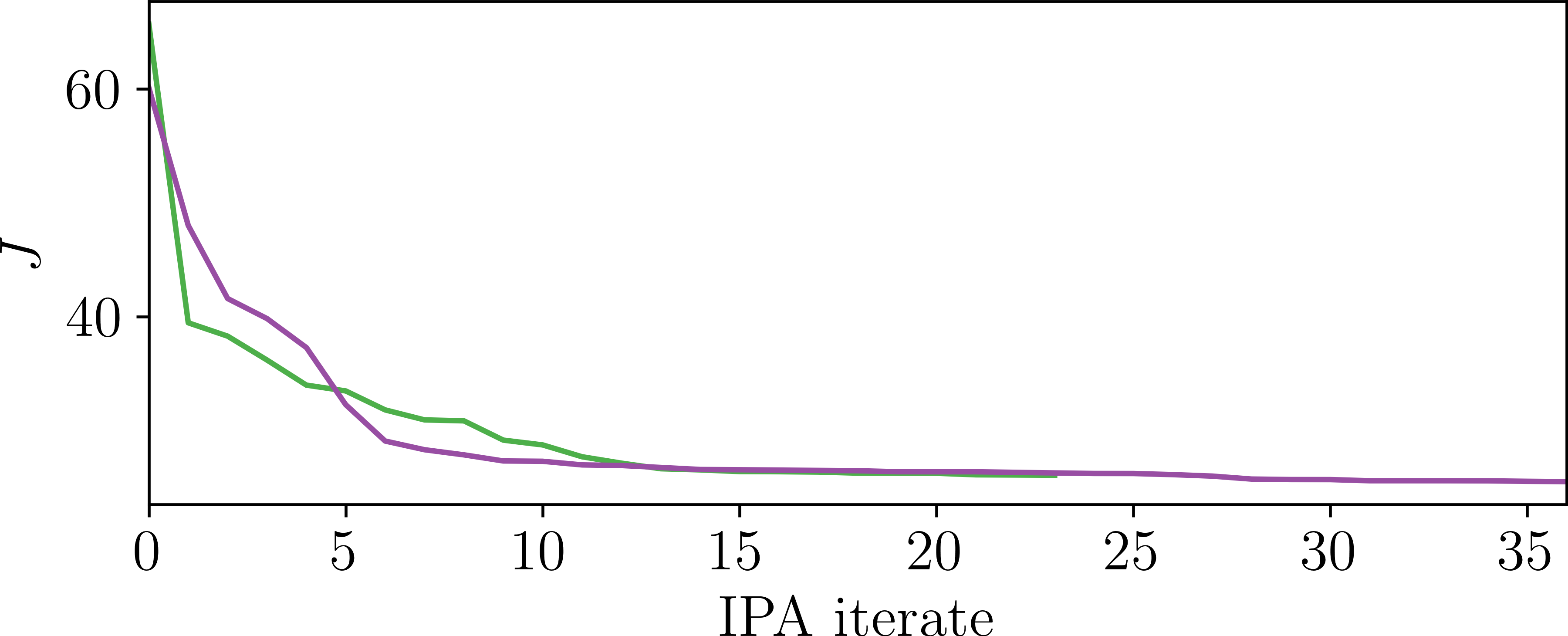}
    \caption{Results for the same experiment as conducted in~\cite{zhou2018optimal}. The top plot shows the comparison of the agent trajectories obtained from the optimal (green) and practical (purple) parameterizations. The targets are depicted in red. The second plot shows the evolution of the target uncertainties, where the line style specifies the target and the color specifies the used parameterization. The third plot shows the cost reduction for the \ac{ipa} iterates.}
    \label{fig:exp:zhou}
\end{figure}
\section{Numerical Experiments}\label{sec:numerical:experiments}
\noindent\textbf{Comparison to existing methods.} We first compare the proposed methods to the related algorithm from~\cite{zhou2018optimal} that was designed for static targets. We apply the new parameterizations to a static target experiment, the globally optimal cost of which was shown to be $J^\ast = 25.07$~\cite{zhou2018optimal}.~\figref{fig:exp:zhou} shows the results for both the optimal and practical parameterizations. The solution quality of the practical ($25.54$) and the optimal ($25.61$) are comparable to the local solution of the previous method ($25.54$), which indicates that the generalized methods remain applicable to static targets as well. The fact that the practical parameterization slightly outperforms the optimal one in this specific case is due to local convergence behaviors, and other initializations may lead to different results.

\noindent\textbf{Simultaneous sensing scenarios.} The theoretical results in this paper were proven under the assumption that only one target could be within an agent's sensing range at a time. In~\figref{fig:exp:non:isolated} we demonstrate the fact that describing the target tracking via convex combinations allows the proposed parameterizations to adapt to scenarios where this assumption is violated. This allows agents to place themselves between targets, which is particularly useful if targets are so close to each other that they share a \textit{deadzone}, i.e., an area in which an agent can be positioned such that the uncertainty of both targets becomes non-increasing. In such a case, the uncertainty of both targets can be driven to zero and kept at that level, which is only possible if the agent stays within that deadzone. 
\begin{figure}
    \centering
    \includegraphics[width=\figwidth]{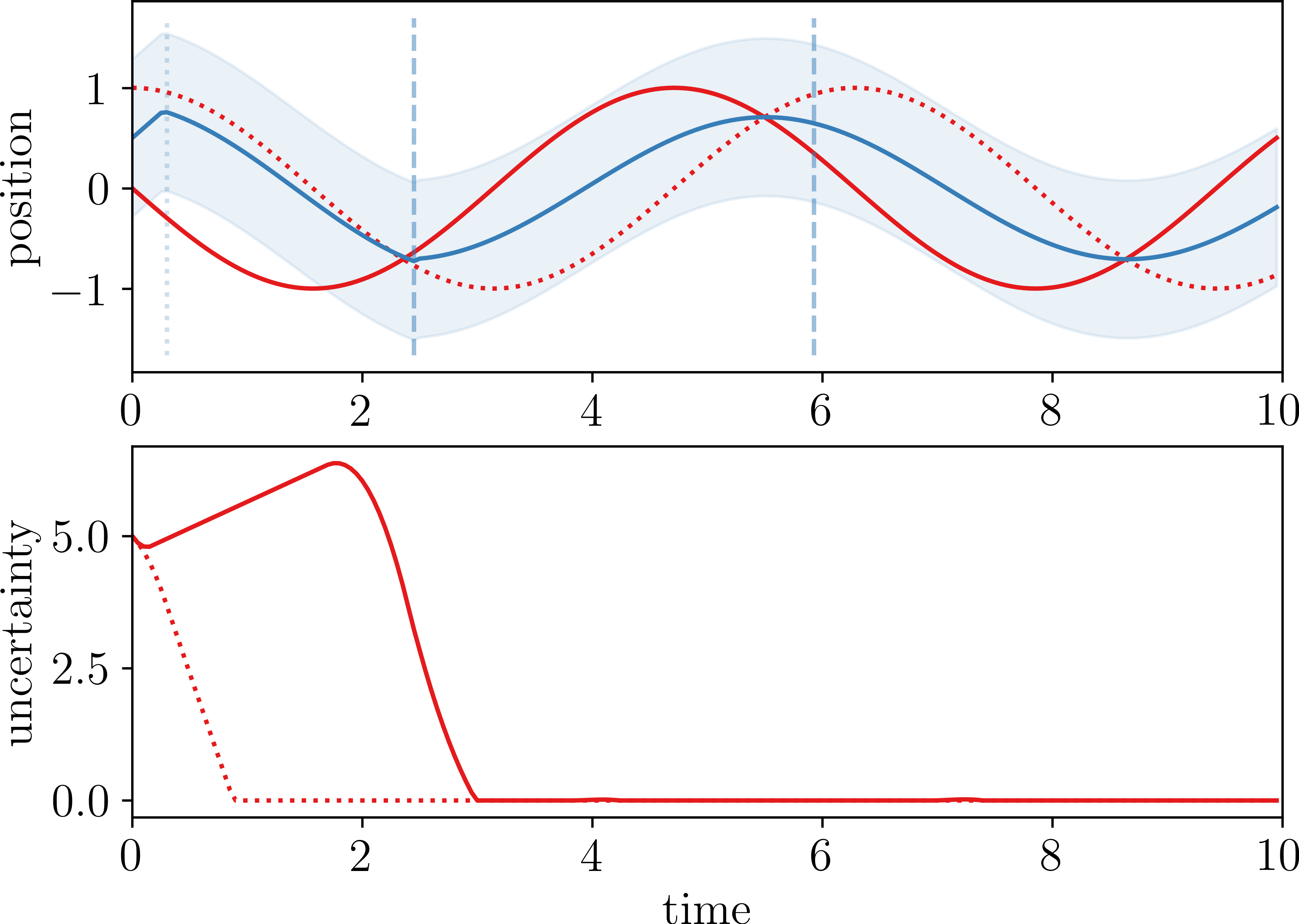}
    \caption{Simultaneous sensing experiment with two targets (red solid and dotted) and one agent (blue solid). The sensing range (blue shaded) is shown on the agent.}
    \label{fig:exp:non:isolated}
\end{figure}

\noindent\textbf{Robustness towards noise.} As previously discussed in~\cite{zhou2018optimal}, the \ac{ipa} analysis provides gradients that are naturally robust towards errors in the estimation of the uncertainty growth rate $A$. This is due to the fact that $A$ only implicitly affects the switching times of the uncertainty dynamics without explicitly appearing in the gradient evaluations.

In this section we want to show that the practical parameterization is additionally robust towards noisy measurements. In particular, we assume that each target's positional measurement is given by $\target_i^{\mathrm{est}}(t) = \target_i(t) + \kappa_1 \nu_{i}(t)$ for $\nu_i(t) \sim \N(0,1)$ and some scaling factor $\kappa_1 > 0$. We ran $1000$ repetitions of a fixed persistent monitoring problem with $M = 4, N = 2, T = 50$, $L = 4$, and randomized initial visiting sequences. We repeated the same experiment for the optimal parameterization with noisy target velocity estimates $\dtarget_i^{\mathrm{est}}(t) = \dtarget_i(t) + \kappa_2 \nu_{i}(t)$, where $\mu_i(t) \sim \N(0,1)$ and $\kappa_2 > 0$.~\figref{fig:uncertainty} compares the cost distributions of the initializations to the resulting trajectories obtained through the respective gradient descent methods. Both methods are able to reduce the cost drastically, though the practical parameterization achieves notably better results. This is due to the fact that the practical version incorporates tracking error feedback into its control law. On the other hand, the optimal version simply tries to copy a target's motions, though without validating its proximity. The practical parameterization's natural resiliency to noise also leads to significantly better initializations, which in turn leads to better local solutions.

\begin{figure}
    \centering
    \includegraphics[width=\figwidth]{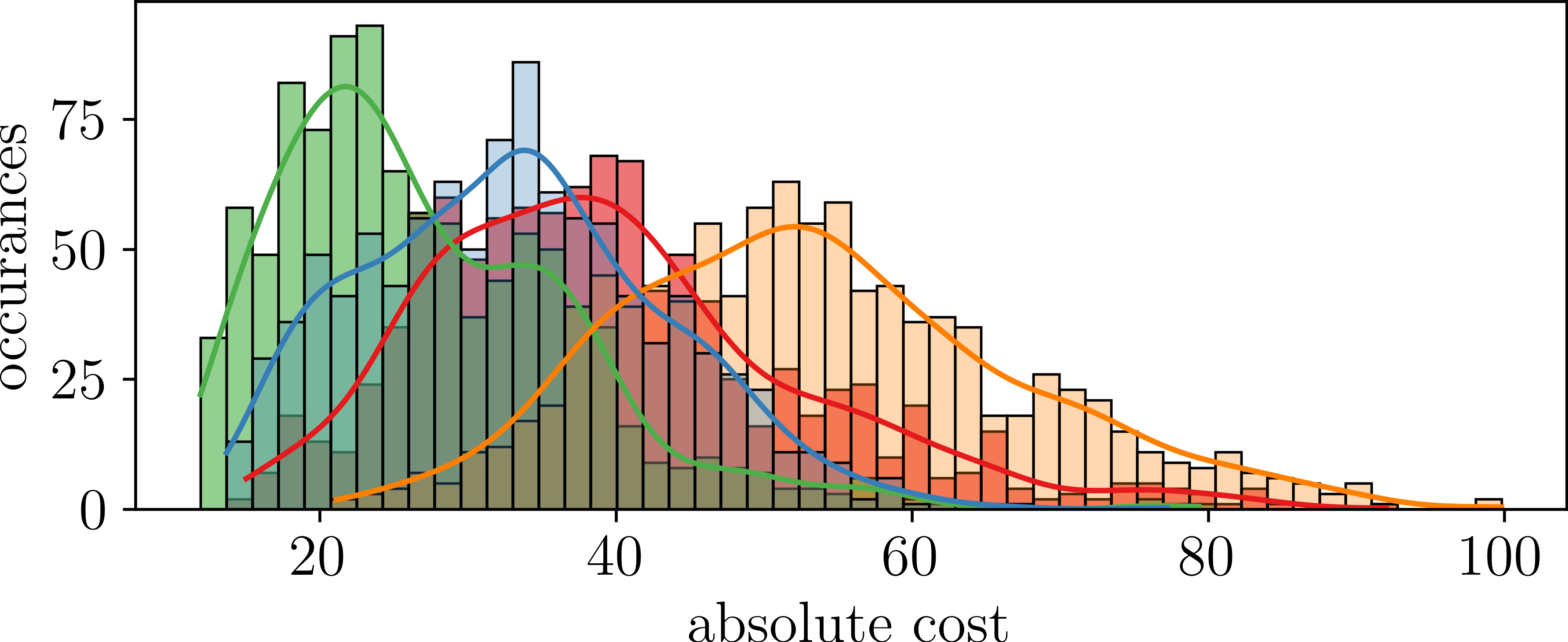}
    \caption{Depiction of the randomized experiment. The cost distribution of the initializations is depicted in orange (optimal) and red (practical), whereas the optimized cost distributions are shown in blue (optimal) and green (practical).}
    \label{fig:uncertainty}
\end{figure}

\section{Conclusion and Future Work}\label{sec:conclusion:future}

In this paper, we showed the existence of optimal solutions for the one-dimensional persistent monitoring problem of mobile targets by reducing the infinite-dimensional \ac{ocp}~\eqref{min:ocp} to a parametric optimization problem. We characterized the optimal control, which motivated the design of an event based optimization scheme. We then proposed a suboptimal parametric description of the system that proved to perform robustly in various numerical experiments.

%The difficulties discussed in~\secref{sec:local:solutions} with respect to finding good initializations show that the design of a scheduler on the higher level remains challenging due to its combinitorial nature. Potential strategies to overcome the addressed problems could consist of homotopy approaches, e.g., by introducing an augmented cost term, varying the sensing range, or varying the nonconvexity induced by the motion of the targets.

In future work, we will consider the extension to more realistic second-order agent dynamics, extending the problem to the infinite horizon case through the use of periodic patterns, and the consideration of problem formulations in two-dimensional settings.

%Space permitting, we can mention that tee nice parameterization in 1D does not carry over to 2D and teh problem is even more challenging

\bibliographystyle{ieeetr}
\bibliography{bibtex/main}

\end{document}